\documentclass[11pt, a4j]{article}
\usepackage{amsmath,amsfonts,amsthm,amssymb,amscd, mathtools}
\usepackage[colorlinks=true, allcolors=blue]{hyperref}

\usepackage{amsthm}
\usepackage{enumitem}
\usepackage{extarrows}
\usepackage{xcolor}
\usepackage{graphicx}
\usepackage{pxrubrica}
\usepackage{multicol}
\usepackage{ascmac}
\usepackage{adjustbox}
\usepackage{xtab}
\usepackage{refcount}
\usepackage{braket}
\usepackage{url}
\usepackage{tikz}
\usepackage{tikz-cd}
\usetikzlibrary{matrix, arrows, decorations.pathmorphing}
\usetikzlibrary{positioning}

\DeclareMathOperator{\id}{id}
\def\dfnref#1{Definition~\ref{dfn::#1}}
\def\thmref#1{Theorem~\ref{thm::#1}}
\def\prpref#1{Proposition~\ref{prp::#1}}
\def\lmmref#1{Lemma~\ref{lmm::#1}}

\def\eqref#1{(\ref{eq::#1})}

\theoremstyle{definition}
\newtheorem{dfn}{Definition}[section]
\newtheorem{thm}[dfn]{Theorem}
\newtheorem{prp}[dfn]{Proposition}
\newtheorem{lmm}[dfn]{Lemma}
\newtheorem{rem}[dfn]{Remark}
\newtheorem{cor}[dfn]{Corollary}
\newtheorem{claim}[dfn]{Claim}

\allowdisplaybreaks[2]

\title{A categorical proof of the nonexistence of $(120, 35, 10)$-difference sets}
\author{
  Hiroki Kajiura\thanks{
     Osaka University of Commerce,
     Mikuriya-Sakae,
     Higashi-Osaka Japan, zip.\ 577-0036.
     \href{mailto:hkajiura@daishodai.ac.jp}{hkajiura@daishodai.ac.jp}
  }
     \and
  Makoto Matsumoto\thanks{
    AMAGAERU Institute of Free Mathematics,
    2-37-6, Narita-Higashi, Suginami-ku Tokyo
    Japan, zip.\ 166-0015.
    \href{mailto:matmotmak@gmail.com}{matmotmak@gmail.com}
  }
}

\date{\today}
\begin{document}
	\maketitle

\begin{abstract}
A difference set with
parameters $(v, k, \lambda)$ is a subset $D$
of cardinality $k$ in a finite group $G$
of order $v$, such that the number $\lambda$
of occurrences of $g \in G$ as the ratio
$d^{-1}d'$ in distinct
pairs $(d, d')\in D\times D$ is
independent of $g$.
We prove the nonexistence of $(120, 35, 10)$-difference sets,
which has been an open problem for 70 years
since Bruck introduced the notion of nonabelian difference sets.

Our main tools are
1.\ a generalization of the category of finite groups
to that of association schemes
(actually, to that of relation partitions),
2.\ a generalization of difference sets to equi-distributed functions
and its preservation by pushouts along quotients,
3.\ reduction to a linear programming in
the nonnegative integer lattice
with quadratic constraints.
\end{abstract}

\section{Introduction}
For the basic terminology on difference sets and block designs, see
a text book \cite{LINT}.
For a finite group $G$ of order $v$, a subset $D\subseteq G$
of cardinality $k$
is said to be a difference set, if and only if,
for any $d\in G$
other than the unit, the cardinality $\lambda$ of
the set
$$
\{
  (g,h)\in G \times G \mid g^{-1} h=d
\}
$$
is independent of $d$.
If $k=0, 1, v-1$ or $v$, the condition trivially holds,
and $D$ is said to be a trivial difference set.
The triple $(v, k, \lambda)$ is called the parameter set of
the difference set. The family of translations
$
gD\subseteq G, g\in G
$
forms a symmetric $(v, k, \lambda)$-block design.
A simple argument shows a strong constraint
$k(k-1)=\lambda (v-1)$.

The notion of difference set is very classical for cyclic groups.
In 1955, Bruck~\cite{difference-set} enlarged the definition
to nonabelian groups as above. The second author
heard a rumor that he mentioned in a talk
on the existence of $(120, 35, 10)$-difference set
in the symmetric group $S_5$ of degree five,
since symmetric groups are most typical
nonabelian finite groups and the
above condition $k(k-1)=\lambda (v-1)$
has a nontrivial solution
for $v=5!$, as the first among factorials
($n!-1$ is a prime for $n\leq 4$).
The existence of nontrivial difference set
in a group of order 120
has been an open problem,
to which this paper answers no.

There are many researches
on the existence of nontrivial difference set,
and most researches use computers.
For example in the case of cyclic groups,
in 1956,
Marshall Hall Jr.\
\cite{DIFF-SET-HALL}
reported that for all $3\leq k \leq 50$,
the existence/nonexistence was decided by using
a computer SWAC, except for 12 cases including
$(120,35,10)$. This parameter set seems to be tough, since
the other 11 cases in cyclic groups are excluded by
\cite{DIFF-SET-MANN, TURYN1, TURYN2}.

In his seminal work \cite{TURYN} on the use of character sums
in 1965,
difference sets with this parameter set
in abelian groups
are excluded (see P.333 there).
For nonabelian groups, in 1978, Kibler \cite{KIBLER}
listed all noncyclic difference sets with $k<20$,
but it does not cover this parameter set.

There are three abelian, 41 nonabelian solvable,
and three nonsolvable groups of order 120,
up to isomorphism.
In 2005, in his doctoral thesis,
Becker~\cite{DIFF-SET-BECKER}
excluded all 44 nonabelian groups of order 120,
except for three solvable groups and the three nonsolvable groups.
As far as the authors know, this is the present state
for the parameter
$(120, 35, 10)$.
(Note that the existence of $(120, 35, 10)$
symmetric block design is still open
\cite[P.52, No. 871]{HANDBOOK-DESIGN}.)
Our method introduced here, with a help of a modest
(note personal) computer and GAP system \cite{GAP}, excluded all these six groups,
leading to the nonexistence of such difference sets.
The method is influenced by the inductive search
along a (fine) sequence of quotient groups, introduced by
Peifer \cite{Peifer} in 2019.
There, he used a chief series
(i.e., a decreasing sequence of normal subgroups
of $G$), and his algorithm does not work well
for such a group as $S_5$ who has
a big simple normal subgroup $A_5$.

Our method is new in the following two points:
1.\ we enlarged the category of finite groups
to the category of unital regular relation partitions
(see Section~\ref{sec:rel-part}),
which is obtained by omitting
important axioms from
the definition of association schemes
\cite{Bannai}
(and our category has the category of
association schemes \cite{HANAKI-CAT, ZIESCHANG}
as a full subcategory, which actually suffices
to run our algorithm),
2.\ we introduced the notion of difference sets
with multiplicities in relation partitions
so that an inductive search works.
Our advantage as an algorithm lies
in the fact that we have much more subgroups
than normal subgroups so that the inductive search
is easier (since $G/H$ for any subgroup $H<G$
is an association scheme called Schurian).

Our result is a good news for both of those
who study association schemes, and those
who use category theory in combinatorics.
For example, Bannai often claimed that association schemes
are ``correct'' generalization of finite groups
(to be exact, that they are ``group theory without groups''), but
it seems that a clear justification is difficult.
The second author knows some combinatorists who claim that
category theory is less useful in combinatorics
than other areas of mathematics. We showed
that the category of association schemes
gives a solution to a 70-years-open problem
after Bruck in combinatorial
group theory, i.e.,
nonexistence of $(120, 35, 10)$-difference sets.

In Section~\ref{sec:diff-set},
we give these generalizations
(even in a broader sense: a difference set
is a special case of multi difference sets,
which are special cases of equi-distributed functions
introduced here).
We remark that a generalization of difference sets
to association schemes is first introduced in
\cite[Definition 5.16]{KAJIURA-MATSUMOTO-OKUDA}.
In Section~\ref{sec:pushout}, we prove the main
Theorem~\ref{thm::pushOutPreserves}
claiming that the pushout of an equi-distributed
function to a quotient is again equi-distributed,
which contains as special cases
\cite[Lemma, P. 469]{difference-set} and
\cite[Lemmas 5 and 6]{Peifer}.
In Section~\ref{sec:algorithm}, we explain
an algorithm for an exhaustive search for
multi difference sets with a given parameter set
based on an induction along quotient maps, using
the theorem above.
By a theorem on the variance (Theorem~\ref{thm:variance})
of an equi-distributed function,
the search is reduced to a Linear Programming
in the nonnegative integer lattice
with (one positive definite and other) quadratic constraints.
\section{Difference set and its generalization}
\label{sec:diff-set}
\subsection{Relation partitions and association schemes}
\label{sec:rel-part}
The following notion is introduced in \cite[Section 2]{KMO}
by Okuda and the authors, which may be seen
as a generalization of
the notion of association schemes,
obtained by omitting some axioms.
We will not use the (important) omitted
axioms in this paper.
\begin{dfn}
Let $X, I$ be sets, and $R:X\times X\longrightarrow I$ a surjection.
We call $(X, R, I)$ a relation partition.
Let $(Y, Q, J)$ be another relation partition.
A morphism from $(X, R, I)$ to $(Y, Q, J)$ is
a pair of functions $f: X\longrightarrow Y$ and $\sigma: I\longrightarrow J$
such that the following diagram commutes:

		\[\begin{tikzcd}
			X\times X \arrow[r,"R"]\arrow[d, "f\times f"'] & I \arrow[d,"\sigma"]\\
			Y\times Y \arrow[r, "Q"']                      & J.
		\end{tikzcd}\]
		It is easy to check that these morphisms constitute the category of relation partitions.
\end{dfn}

\begin{dfn}
A relation partition $(X, R, I)$ is
\textit{regular} if for every $i\in I$, there is an integer $k_i$ such that for any $x\in X$
		\[\#\{x'\in X\mid R(x, x')=i\} = k_i = \#\{x'\in X\mid R(x', x) = i\}\]
holds. The integer $k_i$ is called the $i$--th valency of $(X, R, I)$.
\end{dfn}

\begin{dfn}
A relation partition $(X, R, I)$ is \textit{unital} if there exists
an $i_0 \in I$ such that
\[R^{-1}(\{i_0\}) = \{(x, x) \mid x \in X\}\]
holds.
Such an $i_0$ is unique, if exists,
and is called the \textit{unit} of $(X, R, I)$.
\end{dfn}

\begin{rem}
An association scheme \cite{Bannai} is
a regular unital relation partition.
The category of association schemes studied in
\cite{HANAKI-CAT} and
\cite{ZIESCHANG} is the full subcategory
of the above category of regular unital relation partitions.
Indeed, in this paper, all relation partitions
appeared are association schemes
and
actually Schurian schemes (see \S~\ref{sec:algorithm}).
Thus, it suffices to consider the
category of association schemes
(i.e.\ logically no need to introduce the new category),
but the notion of relation partitions
is much simpler.
\end{rem}
It is easy to see:
\begin{prp}
If $(X, R, I)$ is a unital relation partition,
and $(X', R', I')$ is another unital relation partition,
then any morphism $(f, \sigma)$ from $(X, R, I)$ to $(X', R', I')$
maps the unit $i_0$ to the unit $i'_0$.
\end{prp}

\begin{dfn}
For a relation partition $(X, R, I)$ and $i \in I$,
we define the $i$-th relation by
\[
  R^{-1}(i) := R^{-1}(\{i\}) \subseteq X \times X.
\]
The corresponding adjacency matrix $A_i$ is defined by
\[
A_i(x, x') =
\begin{cases}1 & \text{if}\: R(x, x') = i,\\0 & \text{otherwise}.\end{cases}
\]
Let $g:X\longrightarrow\mathbb C$ be a function.
(The choice of $\mathbb C$ is irrelevant,
for our purpose here
the set of non-negative integers ${\mathbb N}_{\geq 0}$ would suffice).
Then, $g$ may be regarded as a complex vector with index $X$.
The multiplication of $A_i$ and $g$ gives a complex vector indexed by $X$ as follows:
		\[(A_i g)(x') = \sum_{x \in X} A_i(x, x') g(x).\]
The standard Hermitian product is defined by
		\[(g, g') = \sum_{x \in X} g(x) \overline{g'(x)}.\]
\end{dfn}

\subsection{Equi-distributed functions}
\begin{dfn}\label{dfn::equidistribution}
Let $(X, R, I)$ be a unital relation partition,
$g\in \mathbb C^X$ a function, and $v = \#X$.
Define the $i$-th inner distribution $\lambda_i$ of $g$
(this terminology comes from \cite{Delsarte})
by
\[
\lambda_i
:= v\cdot\frac{(A_i g, g)}{\#R^{-1}(i)}
= (A_i g, g)/\frac{\#R^{-1}(i)}{v}.
\]
Here, the meaning of the rational number
$\frac{\#R^{-1}(i)}{v}$
is the average of the number of $i$-th relations over $X$,
which is nothing but $k_i$ if the relation partition is regular
(if we consider $R^{-1}(i)$ as a directed graph over the
vertex set $X$, it is the valency).

We say that $g$ is an \textit{equi-distributed} function
if $\lambda_i$ is independent of the choice of $i\not=i_0$.
We call $g$ a $(v, k, \lambda)$--equi-distributed function,
for
$\lambda = \lambda_i$ ($i\not=i_0$)
and
$k=\sum_{x\in X}g(x)$
($k$ being called the \textit{mass} of $g$).
\end{dfn}

\begin{rem} Let $G$ be a finite group. Then $(G, R, G)$ defined by
	\[R:G\times G\longrightarrow G; \quad (g, h)\longmapsto g^{-1}h\]
is a unital regular relation partition.
This $(G, R, G)$ is called a thin $G$-Scheme
(as an association scheme).
Note that every (possibly noncommutative) finite association scheme
is a unital regular relation partition (a basic fact, see \cite{Bannai}),
and a thin $G$-scheme is an example.
\end{rem}

\begin{dfn}
A function
$\chi:X \to \{0,1\}$
is called the characteristic function for the
subset $\chi^{-1}(1)\subseteq X$.
A function
$\chi:X \to {\mathbb N}_{\geq 0}$
is called a multi characteristic function for the
multi subset of $X$, with the multiplicity
$\chi(x)$ for each element $x \in X$.
These functions are
equivalent notions to the subsets
and the multi subsets of $X$, respectively.
\end{dfn}

\begin{dfn}\label{dfn::classicalDiffSet}
 Let $G$ be a finite group. Let $D \subseteq G$ be a subset.
Let $\delta_D$ be the characteristic function of $D$.
If $\delta_D$ is an equi-distributed function
on $(G, R, G)$, then $D$ is said to be a difference set of $G$.
\end{dfn}

\begin{rem}\label{rem::classicalDS}
The condition that $D$ is a difference set in \dfnref{classicalDiffSet} is interpreted as follows.
For any $i\in G$,
$i\not=\id_G$, $(A_i\delta_D, \delta_D) = \lambda_i$
is independent of
$i\not=\id_G$
(note that $\#R^{-1}(i) = v$,
or equivalently
$k_i=1$, for each $i\in I$).
Since
\[
(A_i\delta_D, \delta_D)
= \sum_{(d, d')\in D\times D}A_i(d, d')
= \#\{(d, d')\in D\times D\mid d^{-1}d'
= i\},
\]
the above definition coincides with the classical definitions of difference sets
as in \cite{difference-set}.
\end{rem}

\begin{dfn} (Multi difference sets.)
A multi characteristic function
$\chi:X \to {\mathbb N}_{\geq 0}$
is called a multi difference set
in
$(X,R,I)$
if and only if
$\chi$ is equi-distributed.
Usually, we consider the corresponding
multi subset of $X$ and call it
a multi difference set
in
$(X,R,I)$.
\end{dfn}

\section{Pushout preserves equi-distributed functions}
\label{sec:pushout}
\begin{dfn}\label{def:quotient} (Quotient of relation partitions.)

Let $(X, R, I)$ be a relation partition,
$(Y, Q, J)$ be another.
Let $f:X \longrightarrow Y$ and $\sigma: I \longrightarrow J$
be surjections that give a morphism of relation partition.
(Note that the surjectivity of $f$ implies that of $\sigma$.)
We say that
$(Y, Q, J)$
is a quotient of
$(X, R, I)$ with quotient morphism
$(f,\sigma)$.
\end{dfn}

The following theorem is a key in our algorithmic search for difference sets.
\begin{thm}\label{thm::pushOutPreserves}
Let $(X, R, I)$ be a unital regular relation partition,
$(Y, Q, J)$ be a unital relation partition.
Let $(f, \sigma)$
be a quotient morphism as above.
For a function $g\in\mathbb C^X$, we define its pushout $f_*g\in\mathbb C^Y$ along $f$ by
	\[f_*g(y) := \sum_{x \in f^{-1}(y)} g(x).\]
If $g$ is $(\#X, k, \lambda)$--equi-distributed function,
then $f_*g$ is a $(\#Y, k, (\#X/\#Y)\lambda)$--equi-distributed function.
\end{thm}

For a proof, we need two lemmas.

\begin{lmm}\label{lmm::lemmaOfPushOutPreserves1}
Let $(X, R, I)$, $(Y, Q, J)$, $f$ and $\sigma$ be as in \thmref{pushOutPreserves}.
Let $A_i$ be the adjacency matrix for $i\in I$, and $B_j$ that for $j\in J$.
Let $g, h\in\mathbb C^X$. Then, it holds that
		\[(B_jf_*g, f_*h) = \sum_{i\in\sigma^{-1}(j)}(A_ig, h).\]
\end{lmm}
\begin{proof}
For $x\in X$, we define $\delta_x\in\mathbb C^X$ by
\[\delta_x(x') = \begin{cases}1&\text{if}\:x=x'\\0&\text{if}\:x\not=x'.\end{cases}\]
Since $\delta_x$ ($x\in X$) spans $\mathbb C^X$, by bilinearity, it suffices to show that
\begin{align}
			(B_jf_*\delta_x, f_*\delta_{x'}) = \sum_{i\in\sigma^{-1}(j)}(A_i\delta_x, \delta_{x'}). \label{eq::ch_func_ver}
\end{align}
It is direct to show that
\[f_*\delta_x = \delta_{f(x)},\]
and hence the left-hand side of \eqref{ch_func_ver} is nothing but
\begin{align}
(B_j\delta_{f(x)}, \delta_{f(x')}) =
   \begin{cases}1 & \text{if } Q(f(x), f(x')) = j, \\
                0 & \text{otherwise}.\end{cases}
\label{eq::lhs_ch_func_ver}
\end{align}
The commutativity in the definition of morphisms implies that
		\[Q(f(x), f(x')) = \sigma\circ R(x, x').\]
In the right-hand side of \eqref{ch_func_ver}, the summation is at most one, and
\begin{align}
	\sum_{i \in \sigma^{-1}(j)} (A_i\delta_x, \delta_{x'})
=
\begin{cases}1 & \text{if}\: R(x, x') = i,\quad i\in \sigma^{-1}(j)\\0 & \text{otherwise}.
\end{cases}\label{eq::rhs_ch_func_ver}
\end{align}
The comparison of \eqref{lhs_ch_func_ver} and \eqref{rhs_ch_func_ver} shows
\lmmref{lemmaOfPushOutPreserves1}.
\end{proof}

\begin{lmm}\label{lmm::lemmaOfPushOutPreserves2}
Assume that a morphism $(f, \sigma)$ is a quotient morphism
from a regular relation partition $(X, R, I)$ to
a unital relation partition $(Y, Q, J)$.
Then, for every $y\in Y$, $\#f^{-1}(y)= \#X/\#Y$ holds.
\end{lmm}
\begin{proof}
For $y\in Y$, take $x_0\in X$ such that $f(x_0) = y$.
It follows that
\begin{align*}
			x\in f^{-1}(y)&\Longleftrightarrow f(x) = y
			               \Longleftrightarrow Q(f(x), y) = j_0
			               \Longleftrightarrow Q(f(x), f(x_0)) = j_0\\
			              &\Longleftrightarrow R(x, x_0)\in\sigma^{-1}(j_0)
\end{align*}
holds. Thus, $\#f^{-1}(y)$ is the summation of valencies of $R_i$, $i\in\sigma^{-1}(j_0)$,
which is independent of the choice of $y$, since every $R_i$ is regular.
\end{proof}

\begin{proof}[Proof of \thmref{pushOutPreserves}]
Suppose that $g\in\mathbb C^X$ is a $(v, k, \lambda)$--equi-distributed function.
Then, for $i\not=i_0$, we have
	\[\lambda = (A_ig, g)\frac{v}{\#R^{-1}(i)}.\]
By \lmmref{lemmaOfPushOutPreserves1},
	\begin{align}
		(B_jf_*g, f_*g)
				&= \sum_{i\in\sigma^{-1}(j)}(A_ig, g) = \sum_{i\in\sigma^{-1}(j)}\lambda\frac{\#R^{-1}(i)}{v} \notag\\
				&= \frac{\lambda}{v}\sum_{i\in\sigma^{-1}(j)}\#R^{-1}(i)\quad (\text{for}\: j\not=j_0)
      \label{eq::f_*gconponentOfB}
	\end{align}
On the other hand,
if $\boldsymbol 1_X\in\mathbb C^X$ denotes
the constant function with value one,
we have by \lmmref{lemmaOfPushOutPreserves1}
\begin{align}
	(B_jf_\ast\boldsymbol 1_X, f_*\boldsymbol 1_X)
    =
    \sum_{i\in\sigma^{-1}(j)}(A_i\boldsymbol 1_X, \boldsymbol 1_X).\label{eq::f_*1_XconponentOfB}
\end{align}
Put $e:=\#Y$.
By \lmmref{lemmaOfPushOutPreserves2},
$f_*\boldsymbol 1_X\in\mathbb C^Y$ is
the constant function with value $v/e$,
and hence
the left-hand side of \eqref{f_*1_XconponentOfB} is
\begin{align}
	\left(\frac{v}{e}\right)^2\#Q^{-1}(j).
\label{eq::lhs_f_*1_XconponentOfB}
\end{align}
The right-hand side of \eqref{f_*1_XconponentOfB} is
\begin{align}
	\sum_{i\in\sigma^{-1}(j)}\#R^{-1}(i).\label{eq::rhs_f_*1_XconponentOfB}
\end{align}
We divide \eqref{f_*gconponentOfB}
by $\eqref{lhs_f_*1_XconponentOfB} = \eqref{rhs_f_*1_XconponentOfB}$
to obtain
\[
   \left(\frac{e}{v}\right)^2\frac{(B_jf_* g, f_* g)}{\#Q^{-1}(j)}
    =
   \frac{\lambda}{v}\quad \text{for}\: j\not=j_0.
\]
Thus we have
		\[e\frac{(B_jf_*g, f_*g)}{\#Q^{-1}(j)} = \frac{v\lambda}{e}.\]
This proves that the $j$-th inner distribution
for $j\not=j_0$ in $(Y, Q, J)$ is constant $v\lambda/e$,
see \dfnref{equidistribution}.
As for the mass of $f_*(g)$, it is equal to
\[\sum_{y\in Y}f_*g(y) = \sum_{x\in X}g(x) = k.\]
Thus, $f_*g$ is an equidistributed function
with parameter $(e, k, v\lambda/e)$.
\end{proof}

\section{An algorithm to search for difference sets}
\label{sec:algorithm}
\subsection{Sequence of quotient morphisms}
Recall that
for a unital regular relation partition $(X, R, I)$,
an equi-distributed function $g \in \mathbb C^X$
is said to be multi characteristic
if the values of $g$
lie in the non-negative integers $\mathbb N_{\geq 0}$.
The corresponding multi subset of $X$ is called
a multi difference set in $(X, R, I)$
(called a difference sum in Peifer \cite{Peifer}
and a quotient image in Becker \cite{DIFF-SET-BECKER}
in the case of
thin schemes).
The following is a direct consequence of Theorem~\ref{thm::pushOutPreserves}.


\begin{lmm}\label{lmm::differenece_tables}
Let $(X, R, I)$ be a unital regular relation partition,
$(Y, Q, J)$ a unital relation partition,
and $(f, \sigma): (X, R, I)\longrightarrow (Y, Q, J)$
a quotient morphism.
Put $v := \#X$ and $e:=\#Y$.
To enumerate all $(v, k, \lambda)$--equi-distributed multi characteristic functions
on $(X, R, I)$,
it suffices to enumerate
all the $(e, k, (v/e)\lambda)$--equi-distributed multi characteristic functions $h$ on $(Y, Q, J)$,
and list their lifts to $(X, R, I)$
(i.e., $g\in\mathbb C^X$ that is multi characteristic and $f_*g = h$),
and check whether $g$ is equi-distributed or not.
Note that the number of
such $g$ is finite because of the nonnegativity,
and hence an exhaustive search is possible.
Also, if $g$ is a characteristic function of
a difference set, then
the value of $h$ is bounded by $v/e$, which
gives a tight condition in the search.
\end{lmm}
This lemma gives an efficient enumeration of difference (multi) sets:
we don't show that, but according to computer experiments,
it is much easier to enumerate difference multi sets
in a quotient than to enumerate in the original
unital regular relation partition, since the number of
variables is small.
In addition, the search in $X$ is divided to
the search in $X$ and the search of lifts,
from which we expect the improvement in efficiency
by ``divide-and-conquer'' principle.

\begin{lmm}[Schurian scheme]\label{lmm::SchurianScheme}
Let $G$ be a finite group, and $H$ its subgroup.
Then the pair $X\coloneqq G/H$, $I\coloneqq H\backslash G/H$ and
\[
R: G/H\times G/H\longrightarrow H\backslash G/H = I;\quad (g_1H, g_2H)\longmapsto Hg^{-1}_1g_2H
\]
give an association scheme (called a Schurian scheme), and hence a unital regular relation partition.
If $K < H$ is a subgroup of $H$, then the pair
\[
G/K\longrightarrow G/H,\quad K\backslash G/K\longrightarrow H\backslash G/H
\]
gives a quotient morphism of association schemes (and hence of unital regular relation partitions).
\end{lmm}

\begin{prp}\label{prp::equi_tables}
Let $G$ be a finite group.
Let $G \gneqq H_1 \gneqq H_2 \gneqq \cdots \gneqq H_n = \{\id\}$ be
a decreasing sequence of finite subgroups.
By \lmmref{SchurianScheme},
we have a sequence of quotient morphisms of unital regular relation partitions
\begin{align*}
	G = G/H_n&\longrightarrow& G/H_{n-1} &\longrightarrow \cdots\longrightarrow& G/H_1&\longrightarrow& G/G\\
	I = H_n\backslash G/H_n &\longrightarrow& H_{n-1}\backslash G/H_{n-1} &
       \longrightarrow \cdots\longrightarrow& H_1\backslash G/H_1 &\longrightarrow& G\backslash G/G.
\end{align*}
By \lmmref{differenece_tables},
We enumerate equi-distributed multi characteristic functions on $G/G$,
and then $G/H_1$, $G/H_2$, $\cdots$ so on, and finally on $G$.
Thus, we can enumerate all the difference set on $G$,
inductively from $G/G$, $G/H_1$, $\ldots$, $G/H_n$, in order.
\end{prp}
Peifer could use only normal subgroups, because he used
only quotient groups, not general Schurian schemes.
This is a big difference, because, for example,
the symmetric group $S_5$ has only one nontrivial normal subgroup $A_5$,
and thus, there is no route other than $S_5>A_5>\{\id\}$, and then
the first step is intractable even in modern computers.
One more remark is that, when $H<G$ is not normal, the following
theorem makes the enumeration down to a linear programming
in the non-negative integer lattice,
with a positive definite quadratic constraint:
\begin{thm} (Variance theorem.)
\label{thm:variance}
Let $(X, R, I)$ be a unital regular relation partition,
and $g \in \mathbb C^X$ be an equi-distributed function
with a parameter set $(v, k, \lambda)$.
Then,
\[
(\sum_{x\in X} g(x))(\overline{\sum_{x\in X} g(x)})
-
\sum_{x\in X} g(x)\overline{g(x)}
=\lambda(v-1)
\]
holds.
Thus, if $g(x)$ is real, the unbiased variance of $g(x)$
is equal to $\lambda$.
\end{thm}
\begin{proof}
Let $J$ denote the $X\times X$-matrix whose components are all one.
Then we have
\[
J=A_{i_0}+\sum_{i \in I\setminus \{i_0\}} R_i.
\]
Application of the Hermitian inner product,
namely
$(g, (-)g)$ to this decomposition of a matrix, gives
an equality
\[
(\sum_{x\in X} g(x))(\overline{\sum_{x\in X} g(x)})
=
\sum_{x\in X} g(x)\overline{g(x)}
+
\sum_{i \in I\setminus \{i_0\}} (R_ig, g).
\]
Because $g$ is an equi-distributed function,
each term in the second sigma at the right equals
$$
\lambda \#R(i)/v
$$
by definition. Since
$$
\sum_{i\in I\setminus \{i_0\}} \#R(i)=v^2-v=v(v-1),
$$
we have the result.
\end{proof}
Note that in the case where $g$ is a multi characteristic function
and the parameters $(v,k,\lambda)$ are given, then,
the above quadratic equality is a strong number-theoretic constraint.
For the case of $g$ being a characteristic function,
this is equivalent to the well-known strong equality for
difference sets:
\[
k(k-1)=\lambda(v-1).
\]
Also remark that,
in the case of the quotient thin group scheme
$G/N$, the condition degenerates to (by putting
$n:=\#N$ and $\#G=nv$ where $v=\#(G/N)$):
\[
k^2-k=(n\lambda)(nv-1),
\]
which trivially follows
from the original constraint
for the existence of a difference set in $G$
with the given parameter. However, for a non-normal
group $H$, this quadratic constraint
on the existence of equi-distributed functions
in $G/H$ with the given parameters is strong
so that an exhaustive search of $g$ with this condition
is often possible and gives an effective cutting branches.
\begin{rem}
It is not obvious that a direct search of difference sets on $G$ is more time-consuming
than inductive search stated in \prpref{equi_tables}.
However, our experiments show that the inductive search
is much more effective.
We again remark that there are preceding researches
by Bruck~\cite{difference-set},
which consider only one-step normal quotient,
and by Peifer~\cite{Peifer},
where the quotient $G/H_i$ is taken only for normal subgroups $H_i$ of $G$,
since these researches use only thin group schemes
(and hence require a chief series),
while we utilize non-normal subgroups (i.e., Schurian association schemes).
Peifer's algorithm works well, but our experiments show that our usage of generalizations
(i.e., unital regular relation partitions)
dramatically accelerates the search,
when there are many more subgroups
than normal ones.
\end{rem}
We said ``integer Linear Programming with quadratic constraints,''
since after fixing an equi-distibuted
function $g$ in $G/H_{i}$, its preimages $g'$
in $G/H_{i+1}$ are enumerated by conditions on
some partial summations of $g'$ decided by $g$,
which is an integer linear programming, and
the above theorem (and other conditions for
being an equi-distributed function) are
quadratic.

We might appeal to simultaneous equations
by enumerating two distinct quotients $G/H$
and $G/H'$, but to eliminate the possibility
of the existence of $(120, 35, 10)$-difference set,
we did not need it.

\subsection{Reducing candidates by equivalence}
As a final comment on computations:
it is much more effective to decrease the
list of multi difference sets by taking
the representatives from its equivalence classes,
as in \cite[Definition 7]{Peifer}.
We explain the notion of the equivalence in our case,
which is essential from the viewpoint of
computation time: for example,
for $G_1$ and $G_3$ (see the next subsection), the number of
candidate multi difference sets
in $G/H_2$ is 45, which is reduced to 2
by taking representatives.
The improvement would be more dramatic,
for the cases with larger value of $MD_2$
(we couldn't continue the enumeration
without taking the representatives
in such cases as stated in the next subsection).

\begin{lmm}
For any subgroup $H<G$ and any
equi-distributed function $g:G/H \to \mathbb C$,
an automorphism
$\phi$ of $G$ fixing $H$ as a subset
and an element $x \in G$,
the function
$g(x\cdot (\phi(-))):G/H \to \mathbb C$
is again an equi-distributed function,
which is said to be an
\textit{equivalent}
equi-distributed
function to $g$.
(Note that this equivalence depends
on $G>H$, not only on $G/H$.)
\end{lmm}
\begin{proof}
It is clear that $\phi$ induces
an equivalent action on
a morphism $G \to G/H$ of relation partitions,
and since the property
``equi-distribution'' depends only on
the structure of relation partitions,
it is preserved by $\phi$.
Since the left translation by $x$
is an equivariant action on the map $R$
(with a trivial action on $I$),
the property is preserved by $x\cdot (-)$.
\end{proof}
\begin{cor}
Let $G>H$ be as above.
Suppose that $h_1$ and $h_2$ are equivalent
equi-distributed functions
in $G/H$.
Then if $g_1$ is an equi-distributed function
on $G$ whose pushout is $h_1$,
there exists an equi-distributed function
$g_2$ on $G$ whose pushout is $h_2$
such that $g_1$ and $g_2$ are equivalent.
\end{cor}
This corollary assures that,
as far as enumerating multi difference sets
with given parameters in a finite
group $G$, it suffices to keep the
representatives of multi difference sets
in a successive quotient
$G/H_i$.
A subtle problem is that,
it may occur that
by taking only representatives in $G/H_i$
some multi difference sets in $G/H_{i+1}$ may be missed.
In other words, we could not prove the following claim.
\begin{claim} (Probably false claim.)

Let $G>H>K$ be a sequence of subgroups.
Suppose that $h_1$ and $h_2$ are equivalent
equi-distributed functions
in $G/H$.
Then if $k_1$ is an equi-distributed function
on $G/K$ whose pushout is $h_1$,
there exists an equi-distributed function
$k_2$ on $G/K$
whose pushout is $h_2$
such that $k_1$ and $k_2$ are equivalent.
\end{claim}
The subtlety is that an automorphism
of $G$ fixing $H$ may not fix $K$. There
is no problem for $K=\{\id\}$.
\subsection{Computations on groups of order 120}
In \cite{DIFF-SET-BECKER}, the
existence of $(120, 35, 10)$ difference sets is extensively studied.
There are 47 groups of order 120, and 44 are nonabelian,
and 38 of them are eliminated in the above study, except
for
three non-solvable groups and three solvable groups.
The three solvable groups are $G_1$, $G_3$ and $G_7$,
which are named after
the groups $[120, 1]$, $[120, 3]$, and $[120, 7]$
in the small groups library of GAP \cite{GAP}.
The three nonsolvable groups are
$S_5$, the direct product $C_2\times A_5$,
and the special linear group over the five
element field
$\mathrm{SL}(2, 5)$ of size 2.

Using the algorithm described above,
we proved that these six groups have no difference sets of this parameter set.
Let $G$ be one of the six groups.
We choose a descending series of subgroups $H_i$
\[
G = G/H_n \twoheadrightarrow
G/H_{n-1} \twoheadrightarrow \cdots
\twoheadrightarrow G/H_1
\twoheadrightarrow G/H_0
\text{ with }
G = H_0,
\]
and enumerate equi-distributed multi characteristic functions on $G/H_i$
that are possible as the pushout of the difference sets
up to equivalence,
and then by \prpref{equi_tables},
we enumerate those on $G/H_{i+1}$ by lifting, starting from $i=0$.
Note that a pushout of a difference set in $G$ to $G/H_i$ has values
at most $\#H_i$, and hence the search is rather restricted.
Obviously the multiplication of all indices
$[H_i:H_{i-1}]$ ($i=1,\ldots,n-1$)
is 120,
and through experiments, it turned out that
to choose these numbers in descending order
(i.e., tough computations at the beginning,
where the number of variables is small
and the computation is easier).
For the six groups, it suffices to choose $n=4$,
and the indices $5, 3$ and $2$,
in other words,
$\#H_i$ is 120, 24, 8, and 4,
for $i=0, 1, 2, 3$, respectively.
No multi difference set with
expected parameter set
exists
in $G/H_3$.
We list these subgroups and the numbers $MD_i$ of
multi difference sets in $G/H_i$ up to the equivalence
described in the previous section.
The cyclic group of order $n$ is denoted by $C_n$.
\begin{enumerate}
\item $G_1 = \Braket{x, y, z\mid y^3, x^5, z^8, xy(yx)^{-1},zx(xz)^{-1},zyz^{-1}y}$,
	\begin{align*}
		H_0 \coloneqq G_1 >
		H_1 \coloneqq \Braket{y, z} >
		H_2 \coloneqq \Braket{z} >
		H_3 \coloneqq \Braket{z^2},
	\end{align*}
with $MD_i$ ($i=0, \ldots, 3$)
1, 2, 2, and 0,

\item $G_3 = \Braket{x, y, z\mid y^3, x^5, z^8, zyz^{-1}y, zxz^{-1}x, yx(xy)^{-1}}$,
	\begin{align*}
		H_0 \coloneqq G_3 >
		H_1 \coloneqq \Braket{y, z} >
		H_2 \coloneqq \Braket{z} >
		H_3 \coloneqq \Braket{z^2}
	\end{align*}
(the same as above, probably by chance, but the authors are not sure),
with $MD_i$ ($i=0, \ldots, 3$) 1, 2, 2, and 0
(again the same as above).
\item $G_7 = \Braket{x, y, z\mid y^3, x^5, z^8, zyz^{-1}y, yx(xy)^{-1}, zxz^{-1}x^{-2}}$,
	\begin{align*}
		H_0 \coloneqq G_7 >
		H_1 \coloneqq \Braket{y, z} >
		H_2 \coloneqq \Braket{z} >
		H_3 \coloneqq \Braket{z^2}
	\end{align*}
(again the same)
with $MD_i$ ($i=0, \ldots, 3$) 1, 8, 5104, and 0.
\item $S_5 = \Braket{(1,2,3,4,5), (1,2)}$,
	\begin{align*}
		&H_0 \coloneqq S_5 >
		H_1 \coloneqq \Braket{(1,2,3,4),\:(1,2)} >
		H_2 \coloneqq \Braket{(1,2,3,4),\:(1,3)} \\
        &\quad >
		H_3 \coloneqq \Braket{(1,2,3,4)},
	\end{align*}
with $MD_i$ ($i=0, \ldots, 3$) 1, 3, 3488, and 0.
\item $C_2\times A_5 = C_2\times\Braket{(1,2,3,4,5), (1,2,3)}$,
	\begin{align*}
		&H_0 \coloneqq C_2\times A_5 \\
        & \quad >
		H_1 \coloneqq \Braket{(1, ()), (0, (2,3)(4,5)), (0, (2,4)(3,5)), (0, (3,4,5))}\\
		&\quad >
		H_2 \coloneqq \Braket{(1, ()), (0, (2,3)(4,5)), (0, (2,4)(3,5))} \\
        & \quad >
		H_3 \coloneqq \Braket{(1, ()), (0, (2,3)(4,5))},
	\end{align*}
where $()$ denotes the identity,
with $MD_i$ ($i=0, \ldots, 3$) 1, 3, 116, and 0.
\item $\mathrm{SL}(2, 5)$,
	\begin{align*}
		&H_0 \coloneqq \mathrm{SL}(2,5) >
		H_1 \coloneqq \Braket{\begin{pmatrix}-1&1\\2&2\end{pmatrix},\:\begin{pmatrix}-1&2\\1&2\end{pmatrix}}\\
		&\quad >
		H_2 \coloneqq \Braket{\begin{pmatrix}0&3\\3&0\end{pmatrix},\:\begin{pmatrix}0&1\\-1&0\end{pmatrix}} >
		H_3 \coloneqq \Braket{\begin{pmatrix}0&1\\-1&0\end{pmatrix}},
	\end{align*}
with $MD_i$ ($i=0, \ldots, 3$) 1, 3, 116, and 0
(again the same numbers as above).
\end{enumerate}
The cputime consumed is
some 20 minutes for $G_1$ and $G_3$,
30 minutes for $C_2 \times A_5$ and
$\mathrm{SL}(2, 5)$,
60 minutes for $S_5$,
and 120 minutes for $G_7$
in MacBookPro18,4
with Apple M1 Max chip,
ten cores and 64GB RAM.
These figures seem to relate
with the number of candidates
shown as $MD_2$ above.

\bibliographystyle{plain}
\bibliography{diffset}
\end{document}